\documentclass{article}
\usepackage{ifpdf}
\ifpdf
   \usepackage[pdftex]{hyperref}
   \usepackage{graphicx}
   \DeclareGraphicsExtensions{.pdf,.jpg}
\else
   \usepackage{graphicx}  
   \usepackage{hyperref}
\fi
\usepackage{graphicx,amssymb,amstext,amsthm,url}
\newcommand{\vek}[1]{\mathbf{#1}}
\newcommand{\Reals}{\mathbb{R}}

\newcommand{\infin}{first-order }
\newcommand{\Infin}{First-order }
\newcommand{\infinly}{first-order }
 \newcommand{\ovtrian}{\widehat{G}}
 
\newtheorem{theorem}{Theorem}
\newtheorem{lemma}{Lemma}

\begin{document}

\title{Combinatorial Characterization of the Assur Graphs from Engineering}
\author{Brigitte Servatius\thanks{ Mathematics Department, 
        Worcester Polytechnic Institute.  bservatius@math.wpi.edu} \and %
        Offer Shai\thanks{Faculty of Engineering, Tel-Aviv University. shai@eng.tau.ac.Il}
        \and Walter Whiteley
        \thanks{Department of Mathematics and Statistics, York University
         Toronto, ON, Canada. whiteley@mathstat.yorku.ca. Work supported
        in part by a grant from NSERC Canada}}

\maketitle

\begin{abstract}
We introduce the idea of Assur graphs, a concept originally
developed and exclusively employed in the literature of the
kinematics community. The paper translates the terminology,
questions, methods and conjectures from the kinematics terminology
for one degree of freedom linkages to the terminology of Assur graphs
as graphs with special properties in rigidity theory.  Exploiting
recent works in combinatorial rigidity theory
we provide mathematical characterizations of these graphs derived
from `minimal' linkages.  With these characterizations, we confirm a
series of conjectures posed by Offer Shai, and
offer techniques and algorithms to be exploited further in future
work.

\end{abstract}

\section{Introduction}
Working in the theory of mechanical linkages, the concept of
`Assur groups' was developed by  Leonid Assur (1878-1920),  a
professor at the Saint-Petersburg Polytechnical Institute. In 1914
he  published a treatise (reprinted in~\cite{Assur}) entitled
{\em Investigation of plane bar mechanisms with lower pairs from
the viewpoint of their structure and classification}. In the
kinematics literature it is common to introduce `Assur groups' (selected groups of
 links) 
as special minimal structures of links and joints
 with zero mobility,
from which it is not possible to obtain a simpler substructure of
the same mobility~\cite{PK}. Initially Assur's paper did not receive
much attention, but in 1930 the well known kinematician
I.I.~Artobolevski\u{i}, a member of the Russian academy of sciences,
adopted Assur's approach and employed it in his widely used
book~\cite{Arto}.  From that time on Assur groups are widely
employed in Russia and other eastern European countries, while their use
in the west is not as common. However, from time to time
Assur groups are reported in  research papers for diverse applications such
as: position analysis of mechanisms~\cite{Mitsi};
finding dead-center positions of planar
linkages~\cite{PK} and others.

The mechanical engineering terminology for linkages (kinematics) and
their standard counting techniques  are introduced via an example in the next section.
Central to Assur's method is the decomposition of complex linkages into
fundamental, minimal pieces whose
analyses could then be merged to give an overall analysis.
Many of these approaches for Assur groups were developed from a range of examples,
 analyzed
geometrically and combinatorially, but never defined with mathematical rigor.

In parallel, rigidity of bar and joint structures as well as motions of related mechanisms
have been studied for several
centuries by structural engineers and mathematicians.
Recently (since 1970) a focused development of a
mathematical theory using combinatorial tools was successful in many applications.
For example for planar graphs there is a simple geometric duality
theory, which, if applied to mechanisms and frameworks yields a relation
between statics and kinematics: any locked planar mechanism is dual to an
unstable planar isostatic framework (determinate truss)~\cite{Shai, CW2}.

The purpose of this paper is twofold. First, we want to draw
together the vocabulary and questions of mechanical engineering
with the rigidity theory terminologies of engineering and
mathematics. Second, we want to apply the mathematical tools of
rigidity theory, including the connections between statics and
kinematics, to give precision and new insights into the
decomposition and analysis of mechanical linkages.

The mathematical tools we need are briefly sketched with references provided in \S2.3-2.7.
Our main result is the description of Assur graphs (our term for Assur groups) in Engineering terms 
(\S2.1,2.2) and 
its reformulation  in mathematical terms.
We show that our mathematical reformulation allows us in a natural way to
embed Assur's techniques in the theory of frameworks (\S3) and bring the results back
to linkages.
In the process we veriy several conjectured characterizations
 presented by Offer Shai in his talk concerning
the generation of Assur graphs and the decomposition of linkages into Assur graphs,
at the 2006 Vienna Workshop on Rigidity and Flexibility \S3.1.  We also give algorithmic processes
for decomposing general linkages into Assur graphs, as well as for generating all Assur graphs (\S3.2-3.3). 

In a second paper~\cite{SSW2}, we will apply the geometric theory of
bar-and-joint framework rigidity in the plane to explore additional
properties and characterizations of Assur graphs.  This exploration
includes singular (stressed) positions of
the frameworks, explored using reciprocal diagrams,
and the introduction of  `drivers', which appear in passing
in the initial example in the next section.

\section{Preliminaries}
In the first two sub-sections we present the mechanical engineering vocabulary,
problems and approaches
through an example.  These offer the background and
the motivation for the concepts of the paper, but do not yet give the
formal mathematical definitions.
In the remaining five sub-sections we give the framework basics needed
 to mathematically describe these approaches.

\subsection{Linkages  and  Assur graphs }
A linkage is a mechanism consisting of rigid bodies, the {\em links},
held together by joints. Since we only consider linkages in the plane,
all our joints are pin joints, or pins. A complex linkage may be efficiently studied by
decomposing it into simple pieces, the Assur groups.  (Engineers use the term group to mean a
specified set of links. In mathematics the word group is used for an algebraic structure, but
most of the tools we will use come from graph theory, so the word graph seems more natural
and we will use it as soon as we start our mathematical section). We introduce these ideas via
an example.

Figure~\ref{tracfig01a} depicts an excavator attached with a linkage system. In the
following, we illustrate how the schematic drawing of this
system is constructed and how it is decomposed into Assur groups.

\begin{figure}[htb]
\centering
\includegraphics{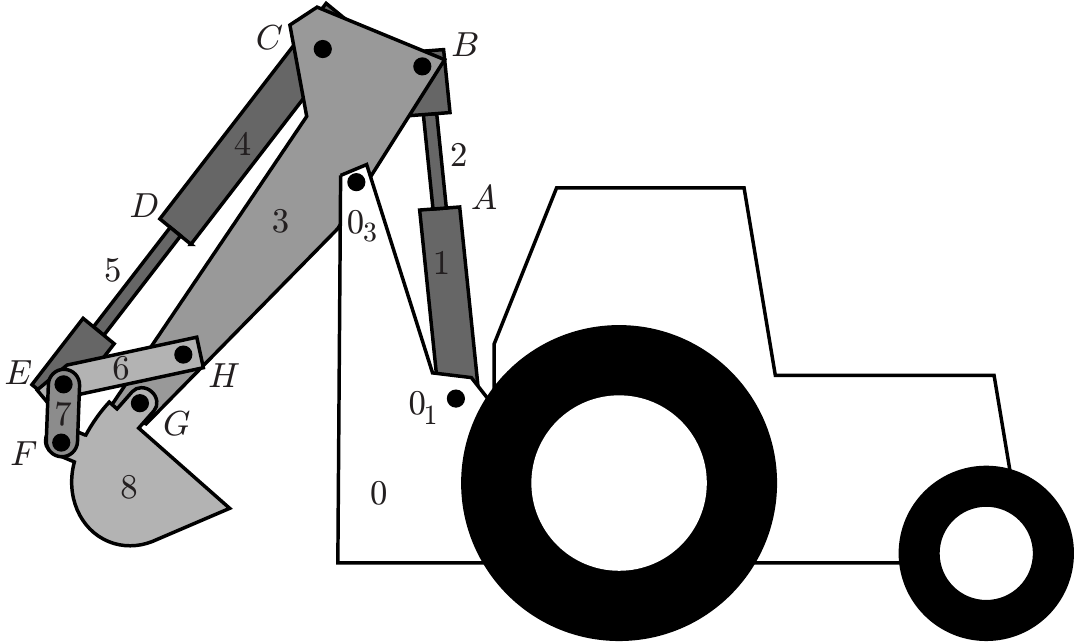}
\caption{The excavator with its kinematic system.\label{tracfig01a}}
\end{figure}

In order to get a uniform scheme, termed {\em structural scheme},  it is
common   to represent all the  connections between the links by
revolute joints as appears in Figure~\ref{tracfig02b}. Here joints~$0_1$~and~$0_3$
attach the excavator to the vehicle (fixed ground) and these special
joints are marked with a small hatched triangle, and are called {\em pinned joints}.
All other joints are called {\em inner} joints. A link
which can be altered  (e.g. by changing its length)
is called a {\em driving link}.
A driving link can be thought of
as driving or changing the distance
between its endpoints like the pistons in our excavator example,
which may be modeled in the
structural scheme by a rotation of an inserted link~1.
\begin{figure}[htb]
\centering
\includegraphics{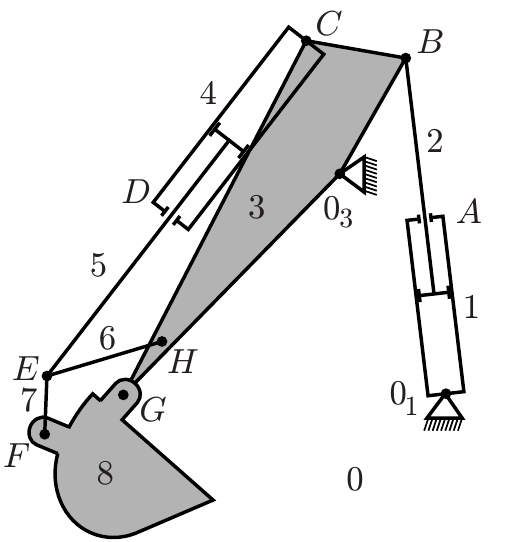}\qquad
\includegraphics{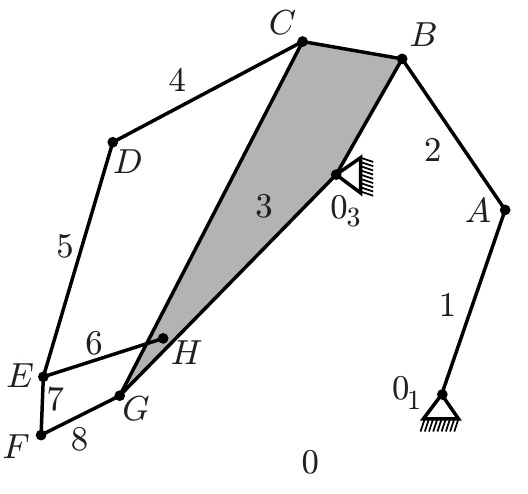}
\caption{The unified structural scheme of the kinematic  system of the excavator.\label{tracfig02b}}
\end{figure}

Once the engineering system is represented in the structural scheme,
to start the analysis, all the driving links are deleted and
replaced by pinned joints (mathematically speaking the driving links are
contracted and their endpoints identified). In the current example, links 1 and 4 are
deleted and joints A and D are pinned. Then, the system is
decomposed into three Assur groups, each consisting of two links, one
inner joint and two pinned joints. In the literature the Assur
groups of this type are referred to as dyads~\cite{Norton}. The order
of the decomposition is important. If an inner joint of a group, $G_1$,
becomes a pinned joint in group $G_2$, then $G_1$ should precede $G_2$.

 In our example (see Figure~\ref{tracfig04}), the unique order of decomposition is:
Dyad$_{1} = \{2,3\}$; Dyad$_{2}=\{5,6\}$; Dyad$_{3}=\{7,8\}$.

\begin{figure}[htb]
\centering
\includegraphics{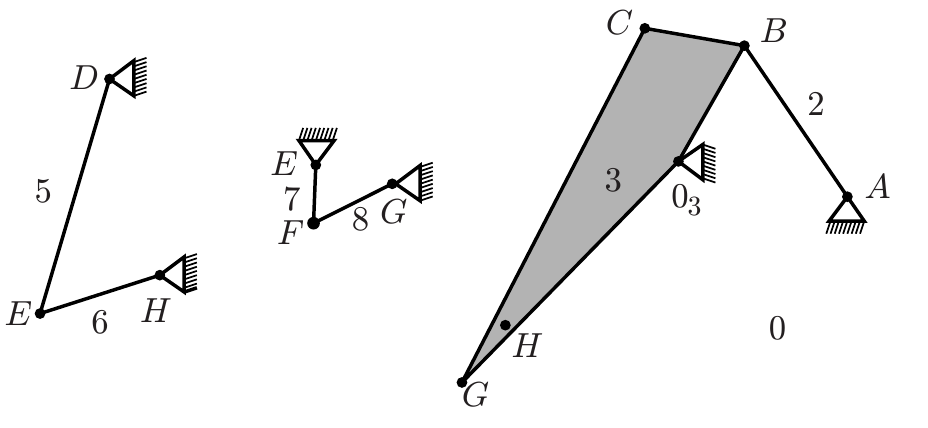}
\caption{Assur group decomposition of
the structural scheme of the kinematic system of the excavator.\label{tracfig04}}
\end{figure}

\subsection{Degree of freedom of a mechanism:  Gr\"{u}bler's equation}
The degree of
freedom (DOF) of a linkage is the number of independent coordinates or measurements
required to define its position.

In mechanical engineering~\cite{Norton}, Gr\"{u}bler's equation
relates the (least number of internal) degrees of freedom $F$ of a
linkage mechanism to the number $L$ of links and the number $J$ of
joints in the mechanism. In the plane, if $J_i$ is the number of
joints from which $i$ links emanate, $i\geq 2$ then
\begin{equation}\label{Grubler1}
F  =3(L-1)-2\sum (i-1)J_i
\end{equation}

In the example above $L=9$ because the fixed ground is considered a
link, and $\sum (i-1)J_i= 11$, because the revolute joint $E$ is
counted twice as it pins links 5, 6, and 7. By Gr\"{u}bler's
equation we get $F=3(9-1)-2\cdot11=2$, which is correct as there are
two driving links (two distances being controled).
 If these two driving links are removed and their
ends pinned (identified) as in the example analysis, we have only 7 links left
and the number of revolute joints is now 9, so $F=0$. This is
another indication that the drivers work independently.

Note that Gr\"{u}bler's equation only gives a lower bound on the degree
of freedom and  there are many cases where the actual DOF is larger
than the predicted one~\cite{Norton}.
If a linkage contains a sub-collection of links pinned in such a way among
themselves that the Gr\"{u}bler count is negative for the sub-collection, then
the predicted DOF for the linkage might be smaller than the actual DOF (see Figure~\ref{miscount}(a)).
This situation can be detected and corrected by combinatorial means as we will describe in
Laman's Theorem, see \S2.5 Theorem~\ref{OverviewTheorem}.

Counting techniques, however, cannot detect special geometries, e.g. parallelism or
symmetry of links, which also might lead to a false  Gr\"{u}bler DOF prediction.
We will examine
this type of geometric singularity of a combinatorially correct graph in~\cite{SSW2}.

\begin{figure}[htb]
\centering
\includegraphics{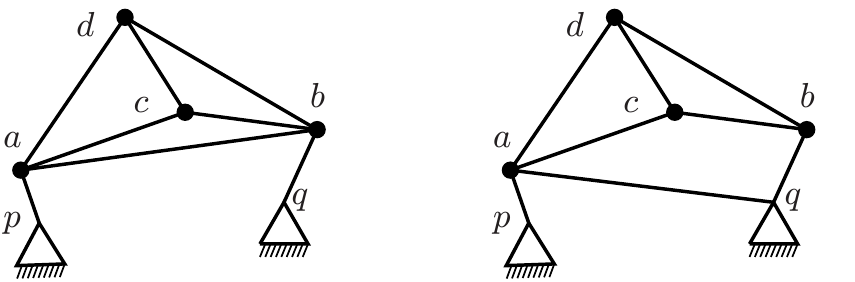}
\caption{In some cases Gr\"{u}bler's equation provides a false answer (a), due to
an overcounted subgraph.  
In others (b) it correctly predicts a pinned isostatic framework (determinate truss).
\label{miscount}}
\end{figure}

\subsection{Frameworks}

For a linkage in which all the links are bars,  with revolute joints at
the two endpoints of the bar, we can rewrite Gr\"{u}bler's equation
in terms of graph theory, by introducing a graph whose vertices, $V$, are
the joints and whose edges, $E$, are the bars. With $V_i$ denoting
the set of vertices of valence $i$,  Gr\"{u}bler's equation
becomes
$$F=3(|E|-1)-2\sum (i-1)J_i=3(|E|-1)-2\sum i|V_i| +2\sum |V_i|$$
$$=
3|E|-3-4|E|+2|V|=2|V|-3 -|E|.$$
So if the edges of a graph embedded in
the plane are interpreted as rigid bars and the vertices as revolute joints,
the graph needs to have at least $2|V|-3$ edges in order to have no
internal degrees of freedom, an observation made already by Maxwell. The count
$2|V|-3$ will
be central to the rest of the paper.

By a {\em framework} we mean a graph $G = (V,E)$ together with a
configuration $\vek{p}$ of $V$ into Euclidean space, for our
purpose the Euclidean plane. We will always assume that the two
ends of an edge (a bar) are distinct points). A motion of the
framework is a displacement of the vertices which preserves the
distance between adjacent vertices, and a framework is {\em rigid}
if the only motions which it admits arise from congruences.

Let us assume that the location $\vek{p}_{i}$ of a vertex is a
continuous function of time, so that we can differentiate with
respect to time. If we consider the initial velocities,
$\vek{p}'_{i}$, of the endpoints $\vek{p}_{i}$ of a single edge
$(i,j)$ under a continuous motion of a framework, then, to avoid
compressing or extending the edge, it must be true that the
components of those velocities in the direction parallel to the
edge are equal, i.e.\
  \begin{equation}\label{InfEqn}
      (\vek{p}_i - \vek{p}_j)\cdot(\vek{p}'_i - \vek{p}'_j) = 0.
  \end{equation}
A function assigning vectors to each vertex of the framework such
that equation~\ref{InfEqn} is satisfied at each edge is called an
{\em \infin  motion.} If the only \infin motions are
{\em trivial}, that is, they arise from \infin  translations
or \infin rotations of $\Reals^2$,  then we say that the framework is
\infinly  rigid in the plane.  \Infin  rigidity
implies rigidity, see for example~\cite{Connelly}.

In our excavator example not all links are bars. In the structural scheme
link 8 is modeled by a bar because it contains only two pins, while link~3,
which contains~5 pins appears as a ``body". We can replace such a body by a
rigid subframework on these~5 vertices  (or more).
In general, any linkage consisting of rigid bodies held together by pin joints
can be modeled as a framework by replacing the bodies with rigid frameworks.

\subsection{The rigidity matrix}

Any graph $G$ can be considered a subgraph of
the complete graph $K_n$ on the vertex set
$V = \{1,\ldots, n\}$, where $n$ is large enough.
Let $\vek{p}$ be a fixed {\em configuration} (embedding) of $V$ into $\Reals^2$.

   Equation~\ref{InfEqn} defines a system of linear equations,
indexed by the edges $(i,j)$, in the variables for the unknown velocities $\vek{p}'_i$. The
matrix $R(\vek{p})$ of this system is a real $n(n-1)/2$ by $2n$
matrix and is called the {\em rigidity matrix}. As an example,
we write out coordinates of $\vek{p}$ and of the
rigidity matrix $R(\vek{p})$, in the case $n=4$.
  \[  \vek{p} = (\vek{p}_{1},\vek{p}_{2},\vek{p}_{3},\vek{p}_{4}) =
        (p_{11},p_{12},p_{21},p_{22},p_{31},p_{32},p_{41},p_{42}); \]
  \[ \left[\begin{array}{cccccccc}
   _{p_{11}-p_{21}} & _{p_{12}-p_{22}} & _{p_{21}-p_{11}} & _{p_{22}-p_{12}} &
   _{0}             & _{0}             & _{0}             & _{0}              \\

   _{p_{11}-p_{31}} & _{p_{12}-p_{32}} & _{0}             & _{0}             &
   _{p_{31}-p_{11}} & _{p_{32}-p_{12}} & _{0}             & _{0}              \\

   _{p_{11}-p_{41}} & _{p_{12}-p_{42}} & _{0}             & _{0}             &
   _{0}             & _{0}             & _{p_{41}-p_{11}} & _{p_{42}-p_{12}}  \\

   _{0}             & _{0}             & _{p_{21}-p_{31}} & _{p_{22}-p_{32}} &
   _{p_{31}-p_{21}} & _{p_{32}-p_{22}} & _{0}             & _{0}              \\

   _{0}             & _{0}             & _{p_{21}-p_{41}} & _{p_{22}-p_{42}} &
   _{0}             & _{0}             & _{p_{41}-p_{21}} & _{p_{42}-p_{22}}  \\

   _{0}             & _{0}             & _{0}             & _{0}             &
   _{p_{31}-p_{41}} & _{p_{32}-p_{42}} & _{p_{41}-p_{31}} & _{p_{42}-p_{32}}

   \end{array}\right]    \]

   A framework
$(V,E,\vek{p})$ is infinitesimally rigid (in dimension $2$) if and
only if the submatrix of $R(\vek{p})$ consisting of the rows
corresponding to $E$ has rank $2n-3$.
We say that the vertex set $V$ is in generic position if the determinant
of any submatrix of $R(\vek{p})$ is zero only if it is identically
equal to zero in the variables $\vek{p}'_i$.
For a generically embedded vertex set, linear dependence of the rows of
$R(\vek{p})$ is determined by the graph induced by the edge set under
consideration. The rigidity properties of a graph are the same for any
generic embedding. A graph $G$ on $n$ vertices  is {\em generically rigid} if the rank
$\rho$ of its rigidity matrix $R_G (\vek{p})$ is $2n-3$, where $R_G (\vek{p})$ is
the submatrix of $R(\vek{p})$ containing all rows corresponding to the edges of $G$,
for a generic embedding of $V$. The (generic) DOF of
$G$ is defined to be $ 2n-3-\rho$.

\subsection{Results for the plane}

Linear dependence of the rows of the rigidity matrix defines a matroid on
the set of rows and for generic configurations  we speak about {\em independent edge sets}
rather than independent rows of $R(\vek{p})$. For a generic embedding of $n$ vertices
into $\Reals^2$ we call the matroid on the complete graph obtained from $R(\vek{p})$
the generic rigidity matroid in dimension $2$ on
$n$ vertices, $\mathfrak{R}_2(n)$.

The following theorem characterizes  $\mathfrak{R}_2(n)$.
\begin{theorem}\label{OverviewTheorem}

           {\em (Laman~\cite{Laman})}
                The independent sets of $\mathfrak{R}_2(n)$ are those sets of edges
                which satisfy Laman's condition:
                \begin{equation}\label{LamansIneq}
                    |F| \leq 2|V(F)| - 3
                    \mbox{~for all $F \subseteq E, F \not = \emptyset $};
                \end{equation}
\end{theorem}

   Laman's Theorem was proved in 1970 and it was this theorem
that promoted the use of matroids to attack rigidity questions. There are many equivalent
axiom systems known for matroids. These can be used to reveal structural properties of various
types and their relationships.  The fact that
matroids are exactly those
structures for which independent sets can
be constructed greedily has important algorithmic conseqences.

From the count condition
in the inequalities~(\ref{LamansIneq}) for independent edge sets it is straight
forward to deduce count conditions for {\em bases} of $\mathfrak{R}_2(n)$ (edge sets inducing
minimally rigid or isostatic graphs), as well as for minimally dependent sets, or {\em circuits},
which will play a fundamental role in our analysis and will be treated in the next section.

Independent sets of $\mathfrak{R}_2(n)$ may be constructed inductively~\cite{TW}.
 Given an
independent edge set $E$ in $\mathfrak{R}_2(n)$, we can extend $E$ by new edges provided that
the inequalities~\ref{LamansIneq} are not violated.  Starting with an independent
set (e.g. a single edge):
\begin{enumerate}
\item[(a)] We can attach a new vertex $v$
by two new edges $x=(v,u)$ and $y=(v,w)$ to the subgraph of $G$
induced by $E$ and $E\cup\{x,y\}$
is independent, see Figure~\ref{extendfig01}a.  This is also called {\em $2$-valent vertex
addition}.
\item[(b)]
Similarly, we can attach a new vertex $v$ by three new edges to the endpoints of an edge $e \in E$ plus any other
vertex in the subgraph of $G$ induced by $E$, and ${E\setminus e}\cup\{x,y,z\}$ is independent,
see Figure~\ref{extendfig01}b. This
operation is called {\em edge-split}, because the new vertex $v$ is
thought of as splitting the edge $e$.
\end{enumerate}
\begin{figure}[htb]
\centering
\includegraphics{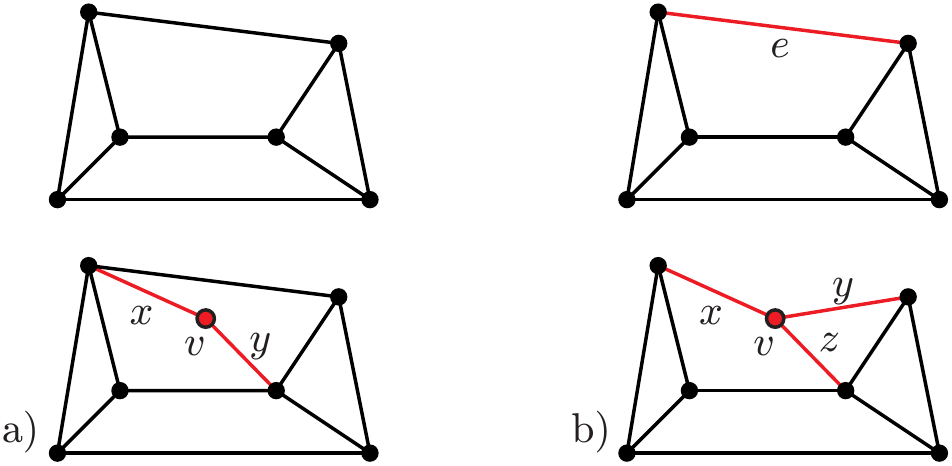}
\caption{Building up an independent set of edges by:
$2$-valent vertex~(a) and edge-split~(b)\label{extendfig01}}
\end{figure}
These {\em Henneberg techniques} developed in~\cite{TW} have become standard in rigidity theory, see also~\cite{GSS}, and
when we resort to ``the usual arguments'' within some of the proofs to come, we have these standard
proof techniques in mind. For further reference we state the following well known result.

\begin{theorem}[Henneberg~\cite{TW}]
Any independent set in $\mathfrak{R}_2(n)$ can be obtained from a single edge by a sequence
of $2$-valent vertex additions and edge-splits.\label{Henneberg}
\end{theorem}
Figure~\ref{henny01} illustrates a sequence described in Theorem~\ref{Henneberg}.
\begin{figure}[htb]
\centering
\includegraphics{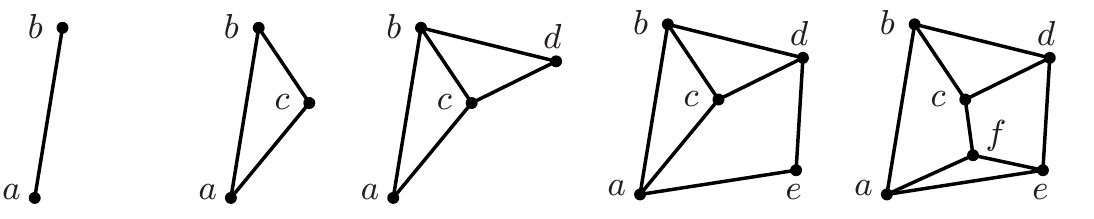}
\caption{A Henneberg sequence constructing an isostatic graph.\label{henny01}}
\end{figure}
\subsection{Rigidity circuits}
Minimally dependent sets, or circuits, in $\mathfrak{R}_2(n)$ are edge sets
$C$ satisfying
$|C|=2|V(C)|-2$ and every proper non-empty subset of $C$ satisfies inequality (\ref{LamansIneq}).
Note that these circuits, called {\em rigidity circuits}, always have an even number of edges.
We will, as is commonly done, not distinguish between edge sets and the graphs they induce.

  Similarly to the inductive constructions of independent sets (see Figure~\ref{extendfig01}), all
circuits in $\mathfrak{R}_2(n)$ can be constructed from a tetrahedron (the complete graph on four
vertices) by two simple
operations, see~\cite{Berg-Jordan}, namely {\em edge-split } as in Figure~\ref{extendfig01}b,
and {\em 2-sum}, where
the 2-sum of two (disjoint) graphs is obtained by ``gluing'' the graphs along an edge
and removing
the glued edge, see Figure~\ref{twosum}.
\begin{figure}[htb]
\centering
\includegraphics{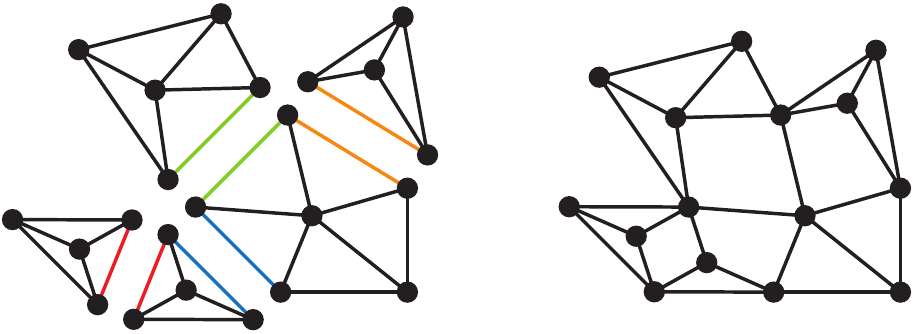}
\caption{2-sums taken along the lined up edge pairs  combine circuits into a
larger circuit. \label{twosum}}
\end{figure}

\begin{theorem}[Berg and Jordan~\cite{Berg-Jordan}]
Any circuit in $\mathfrak{R}_2(n)$ can be obtained from $K_4$ by a sequence
edge-splits and 2-sums.\label{BJ}
\end{theorem}

\subsection{Isostatic Pinned Framework}
Given a framework associated with a linkage,  we are interested in its internal motions,
not the trivial ones, so following the mechanical engineers we pin the framework
 by prescribing, for example, the coordinates of the endpoints of an
edge, or in general by fixing the position of the vertices of some rigid subgraph, see Figure~\ref{pinsub01}. We call
these vertices with fixed positions {\em pinned}, the others {\em inner}. (Inner vertices are sometimes called
{\em free} or {\em unpinned} in the literature.)
Edges among pinned vertices are irrelevant to the analysis of a pinned framework.
We will denote a pinned graph by $G(I,P;E)$, where $I$ is the set of inner vertices, $P$ is the
set of pinned vertices, and $E$ is the set of edges, where each edge has at least one endpoint in $I$.

A pinned graph $G(I,P;E)$ is said to satisfy the
{\em Pinned Framework Conditions} if $|E|=2|I|$ and for all subgraphs
$G'(I', P'; E')$ the following conditions hold:
\begin{enumerate}
\item          $|E'|\leq 2|I'|$ if $|P'| \geq 2$,

\item          $|E'|\leq 2|I'| -1$ if $|P'|=1$ ,  and

\item          $|E'|\leq 2|I'|-3$ if $P' = \emptyset$.
\end{enumerate}

We call a pinned graph $G(I,P;E)$ {\em pinned isostatic} if $E=2|I|$ and $\widetilde{G}= G\cup K_{P}$ is rigid
as an unpinned graph, where $K_{P}$ is a complete graph on a vertex set containing all pins (but no
inner vertices).
In other words,
we ``replace'' the
pinned vertex set by a complete graph containing the pins and call $G(I,P;E)$ isostatic, if choosing any basis in that
replacement produces an (unpinned) isostatic graph.

A pinned graph $G(I,P;E)$ realized in the plane, with $\vek{P}$ for the pins, 
and $\vek{p}$ for all the vertices, is a {\em pinned framework}. A pinned framework is
{\em rigid} if the matrix $R_{\widetilde{G}}$ has rank $2|I|$, where $R_{\widetilde{G}}$
 is the rigidity matrix of $\widetilde{G}$
with the columns corresponding to the vertex set of $K_p$ removed,
{\em independent} if the rows of $R_{\widetilde{G}}$ corresponding to $E$ are  independent,
and {\em isostatic}, if it is rigid and independent. The vertices $I$ of a pinned
framework are in {\em generic position} if any submatrix of the rigidity matrix is zero only if it is identically
equal to zero with the coordinates of the inner vertices as variables.
The coordinates of the pins are prescribed
constants.

Figure~\ref{pinsub01} shows an example of a pinned isostatic $G$ and a corresponding basis 
$\ovtrian$ of $\widetilde{G}$.
\begin{figure}[htb]
\centering
\includegraphics{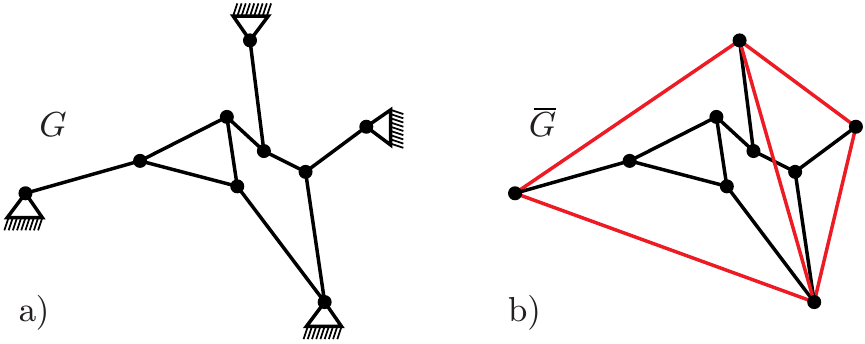}
\caption{Framework (a) is pinned isostatic because Framework (b) is isostatic.\label{pinsub01}}
\end{figure}

 It is common in engineering to choose pins
in advance and their placement $\vek{P}$ might not be generic,
in fact not even in {\em general position} (no three points collinear), as it is sometimes 
necessary to have all pins on a line. The following result shows that this is not a problem.

\begin{theorem}\label{PinThm}
Given a pinned graph $G(I,P;E)$, the following are equivalent:

(i) \ \ \ There exists an isostatic realization of $G$.

(ii) \ \ The Pinned Framework Conditions are satisfied.

(iii) \ For all placements $\vek{P}$ of $P$ with at least two distinct locations and all generic
positions of $I$ the resulting pinned framework is isostatic.
\end{theorem}

\begin{proof}
Let $\widetilde{G}= G\cup K_{P'}$
$P'\supseteqq P$, and $F$ a maximal independent edge set of $K_{P'}$.
Then by Theorem~\ref{OverviewTheorem} we deduce that $E\cup F$ is
isostatic if and only if the Pinned
Framework Conditions are satisfied, so $(i)\Leftrightarrow (ii)$

In order to show $(ii)\Rightarrow (iii)$ we first show that we can
extend $F$ to $F\cup E$ by a Henneberg sequence of 2-valent vertex
additions and edge-splits (see Figure~\ref{extendfig01}). To this
end we de-construct $\widetilde{G}$ first by removing inner vertices as
follows. Assume that $|I|>1$. Since $|E|=2|I|$ and $I$ spans at
most $2|I|-3$ edges, we must have at least three edges joining the
set of inner vertices $I$ to the pinned vertices $P$.  Therefore,
if we sum over the valence of inner vertices, and denote the the
set of edges with both endpoints in $I$ by $E_i$, the ones with
one endpoint in $I$ by $E_p$, we obtain $\sum val(i) = 2|E_i|
+|E_p| = 2|I|+|E_i| \leq 4|I|-3$.   So there will be some inner
vertices of valence 2 or 3.   If at this stage there is some
vertex of degree $2$, we can just remove it, to create a smaller
graph with the same isostatic count.   If there is some vertex of
degree $3$, then by the usual arguments~\cite{GSS, TW}, it can be
removed, and replaced by a new edge joining two of its neighbors,
which were not yet joined in a remaining rigid subgraph (e.g.\ not
both pinned), to create a smaller subgraph with the isostatic
count.  This produces a reverse sequence of smaller and smaller
isostatic graphs until we have no inner vertices.

   To obtain a realization, place $P$ in an arbitrary position $\vek{P}$
  with at least two distinct
vertices. Create an isostatic graph whose vertex set contains $P$, by, for example,
ordering the vertices in $P$ with distinct positions arbitrarily, $P=\{ p_1, p_2, \ldots \}$, adding edges between
consecutive vertices and attaching an extra vertex, ${p}_0$ by edges $(p_0,p_i)$. This graph
is clearly isostatic, provided the point $\vek{p}_0$ is not placed on the line through $\vek{p}_i$ and
$\vek{p}_{i+1}$ for any $i$,  because it has the correct edge count and is rigid since it consists
of a string of non-collinear triangles.

To complete the proof, we work back up the sequence of subgraphs we created in the de-
construction process.  We assume the current graph is
realized as isostatic.  When the next graph is created by adding a 2-valent vertex,
then adding such a vertex in any position except on the line joining its two
attachments gives a new isostatic realization.

When the next graph is created by an edge-split,  note that at
least one of the neighbors of the new vertex is inner, so this added inner vertex can
be placed in a generic position ensuring that the three new
attachments are not collinear.  Therefore, by the usual
arguments~\cite{GSS, TW}, this insertion is also isostatic when
placed along the line of the bar being removed, and therefore also
when placed in any generic position. Since (iii) trivially implies
(i), the proof is complete.
\end{proof}

A pinned graph $G(I, P;E)$ satisfying the Pinned Framework Conditions must have at least two
pins and in every isostatic realization of $G$ there must be at least two distinct pin locations.
Placing all pins in the same location never yields an isostatic framework, but we can make an
important observation about the DOF of such a ``pin collapsed'' framework.

\begin{theorem}\label{contractThm}
Let $G(I, P; E)$ be a pinned graph satisfying the Pinned Framework Conditions. Identifying the
pinned vertices to one vertex $p^{*}$ yields a graph $G^{*}(V,E)$, $V=I\bigcup \{p^{*}\}$ and 
the DOF of $G^{*}$
is one less than the number of rigidity circuits contained in $G^{*}$.
\end{theorem}

\begin{proof}
Since $|E|=2|I|=2|V|-2$, $G^{*}$ contains too many edges to be isostatic. If $G^{*}$ is rigid, it is overbraced
by exactly one edge, hence contains exactly one rigidity circuit. If $G_p$ is not rigid, each of
the rigidity circuits in $G^{*}$ must contain $p^{*}$. If two rigidity circuits intersected in a vertex
other than $p^{*}$, the union of their edge sets, together with the pinned subgraph would violate the
Pinned Subgraph Conditions. Therefore all circuits in $G^{*}$ have exactly the vertex $p$ in common.
Removing exactly one edge from each circuit yields a basis for $G^{*}$ in $\mathfrak{R}(G^{*})$, 
establishing
the desired connection between the DOF and the number of circuits.
\end{proof}

\section {Characterizations of Assur graphs}
We start with two citations from the mechanical engineering literature as motivation for our combinatorial
conditions.
The following definition appears in~\cite{YV}: ``An Assur group is obtained from a kinematic
chain of zero mobility by suppressing one or more links, at
the condition that there is no simpler group inside". In~\cite{SS} we find:
``An element of an Assur group is a kinematic chain with free or unpaired joints
on the links which when connected to a stationary link will have zero DOF.
A basic rigid chain is a chain of zero DOF and whose subchains all have
DOF greater than zero. In other words an element of an Assur group is a basic rigid
chain with one of its links deleted".

These descriptions  from the engineering literature are not definitions in the mathematical sense,
but rather use `minimality' informally, as in the original work of Assur.
We are now ready to give a formal definition by confirming  a series of equivalent combinatorial
characterizations of Assur graphs.
These statements are new, and (iii) and (iv) come from the conjectures offered by
Offer Shai at the workshop.

\subsection{Basic Characterization of Assur graphs}

\begin {theorem}\label{CharacterThm}
Assume $G= (I,P; E)$ is a pinned isostatic graph.  Then the following are equivalent:

(i) $G= (I,P; E)$  is minimal as a pinned isostatic graph:  that is for all
proper subsets of vertices  $I'\cup P'$,  $I' \cup P'$ induces a pinned subgraph
$G' = (I'\cup P',  E')$ with  $|E'| \leq  2|I'| -1$.

(ii)     If the set $P$ is contracted to a single vertex $p^*$, inducing the unpinned
graph $G^*$ with edge set $E$, then $G^*$ is a rigidity circuit.

(iii)   Either the graph has a single inner vertex of degree~$2$ or each time we
delete a vertex, the resulting pinned graph has a motion of all inner vertices
(in generic position).

(iv)  Deletion of any edge from $G$ results in a pinned graph that has a motion
of all inner vertices
(in generic position).
\end {theorem}

\begin{proof}
(i) implies (iv)   If we delete an edge, there must be a
motion by the count. If there is a set of inner vertices that are not moving,
    in generic position, then these vertices, and their edges to the pinned vertices,
    must form a proper isostatic pinned subgraph contradicting condition (i).

    (iv) implies (iii) Removing an edge with an endpoint of valence $2$ produces a graph
    with a pendant edge. This must be the only inner vertex, since any other inner vertices
    are not moving, contradicting (iv). Since removing a single edge results in a motion of
    all inner vertices, removing all edges incident with one particular vertex results in a framework
    with a motion on all the remaining vertices.



  Conversely, (iii) implies (i)  If the graph contains a minimal proper pinned
  subgraph, then removing any vertex outside of this subgraph will produce a motion
  at most in the vertices outside of the subgraph.  This contradicts (iii).


(i) is equivalent to (ii)  If $G= (V,P; E)$ is a pinned isostatic graph, then identifying
the vertices in $P$ to a single vertex $p*$ yields a graph $G^{*} = (V^{*}, E^{*})$ with
$|E^{*}| = 2(|V|+1) = |V^{*}|-2$, so $G^{*}$ is dependent and if the minimality condition in (i) is
satisfied, it must be a rigidity circuit. Conversely, if $G^{*}$ is a rigidity circuit, we can
pick an arbitrary vertex of $G^{*}$ and call it $p^{*}$. Splitting $p^{*}$ into a vertex set $P$,
$|P|\geq 2$, (where $P$ may have as many vertices as
the valence of $p^{*}$ in $G^{*}$ allows) and specifying for each edge with endpoint $p^{*}$ a
new endpoint from $P$ so that no isolated vertices are left, yields an isostatic pinned framework
satisfying the minimality condition.
\end{proof}

This theorem provides a rigorous mathematical definition: an {\em Assur graph}
is a pinned graph satisfying one of the four equivalent conditions in Theorem~\ref{CharacterThm}.

Condition (i) is
a refinement of the Gr\"{u}bler count~(\ref{Grubler1}), in a form which is  now necessary and sufficient.

Condition (ii) translates the minimality
condition to minimal dependence in $\mathfrak{R}_2 (n)$ and thus serves as a purely combinatorial
description of Assur graphs and may be checked by fast
algorithms~\cite{JacobsHendrickson, pebblegame}.

Conditions (iii) and (iv) are similar in nature. Condition (iii)
provides the engineer with a quicker check for the Assur property
for smaller graphs than (iv), since there are fewer vertices than
edges to delete. However,  condition (iv) tells the engineer that
a driver inserted for an arbitrary edge will (generically) move
all inner vertices.  We will expand on this property in~\cite{SSW2}
\begin{figure}[htb]
\centering
\includegraphics{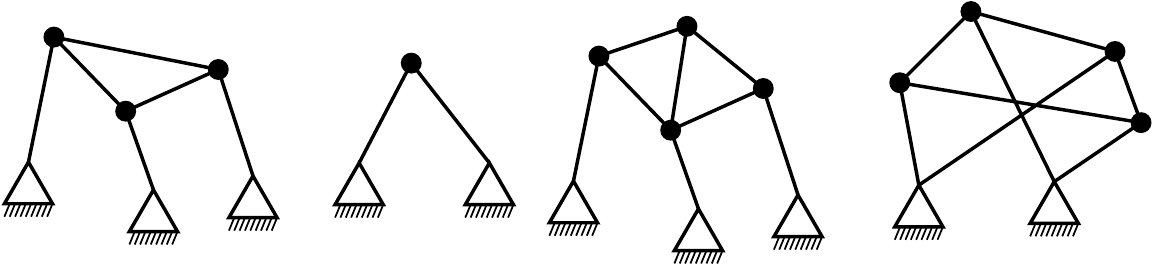}
\caption{Assur graphs\label{assurwordfig01}}
\end{figure}
\begin{figure}[htb]
\centering
\includegraphics{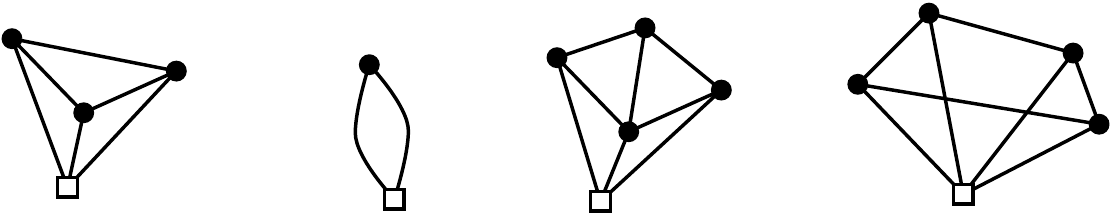}
\caption{Corresponding circuits for Assur graphs\label{assurwordfig03}}
\end{figure}

Some examples of Assur graphs are drawn in Figure~\ref{assurwordfig01} and their corresponding rigidity
circuits in Figure~\ref{assurwordfig03}.

\subsection{Decomposition of general
isostatic frameworks }

We  now show that a general isostatic framework can be decomposed into a
partially ordered set of Assur graphs. The given framework can be re-assembled from
these pieces by a basic linkage composition.
Figure~\ref{assurwordfig02} shows isostatic pinned frameworks and
 Figure~\ref{assurwordfig05}
indicates their decomposition.
\begin{figure}[htb]
\centering
\includegraphics{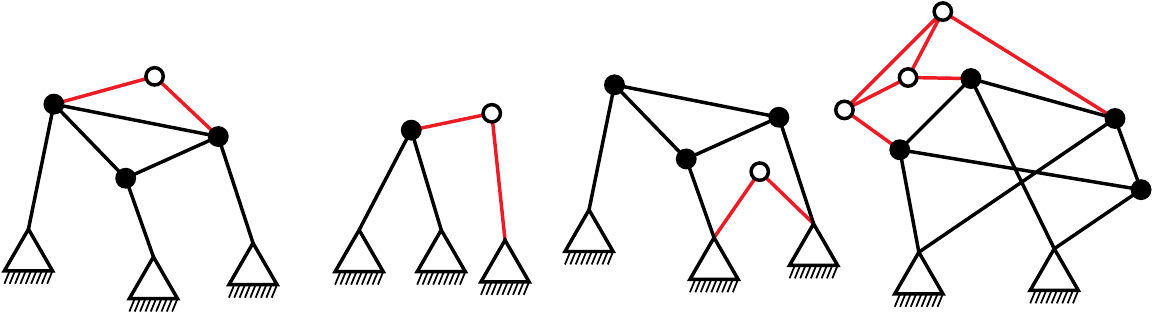}
\caption{Decomposable - not Assur graphs\label{assurwordfig02}}
\end{figure}

\begin{figure}[htb]
\centering
\includegraphics{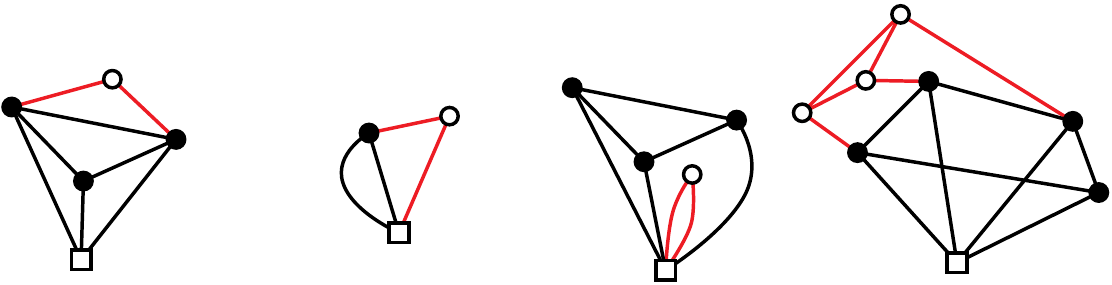}
\caption{The first step of a decomposition for isostatic
frameworks in \ref{assurwordfig02} - with identified
subcircuit(s).\label{assurwordfig04}}
\end{figure}

Given two linkages as pinned  frameworks $H= (W,Q;F)$ and $G= (V,P; E)$ and an injective map
$C: Q \rightarrow V\cup P$, the
{\em linkage composition} $C(H,G)$ is the linkage obtained from $H$  and $W$ by identifying
the pins $Q$ of $H$ with their images $C(Q)$.

\begin{lemma} Given two pinned isostatic graphs $H= (W,Q,F)$ , $G= (V,P; E)$,
the composition $C(H,G)$ creates the new composite pinned graph:
$C=(V\cup W, P, E\cup F)$ which is also isostatic.
\end{lemma}

\begin{proof}  By the counts, we have $|F|=2|W|$, and $|E|=2|V|$, so
$|E\cup F| = 2|W \cup V|$.   A similar analysis of the subgraphs confirms the isostatic status.
\end{proof}

\begin{figure}[htb]
\centering
\includegraphics{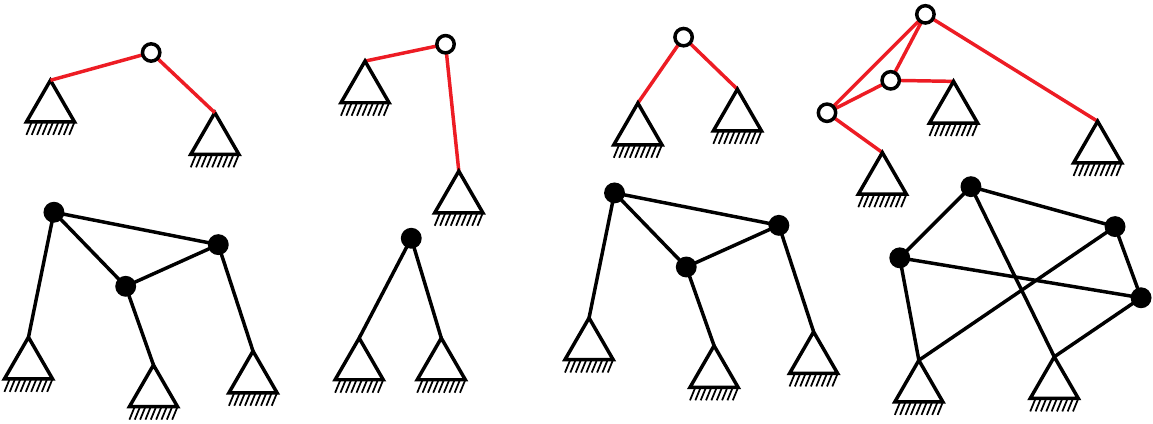}
\caption{Recomposing the pinned isostatic graphs in Figure~\ref{assurwordfig02}
from their Assur components.\label{assurwordfig05}}
\end{figure}

Under this operation, the Assur graphs will be the minimal, indecomposable graphs.  We
can show that every pinned isostatic graph $G$ is a unique composition of Assur graphs, which we will
call the {\em Assur components} of $G$.

\begin{theorem} A pinned graph is isostatic if and only if it decomposes \label{POTheorem}
into Assur components. The decomposition into Assur components is unique.
\end{theorem}

\begin{proof}  Take the isostatic pinned framework, and identify the ground pins.
This is now a dependent graph.  Using properties of $\mathfrak{R}_2 (n)$, see~\cite{CMW}, we can 
identify minimal dependent subgraphs - which, by
Theorems~\ref{contractThm} and~\ref{CharacterThm}, 
are Assur components after the pins are separated. These are the initial components. When all
of these initial components are contracted in step two, we seek additional Assur components.
We iterate the process until only the ground is left.
\end{proof}

The decomposition process described in the proof of
Theorem~\ref{POTheorem} naturally induces a partial order on the
Assur components of an isostatic graph:  component $A\leq B$ if
$B$ occurs at a higher level, and $B$ has at least one vertex of
$A$ as a pinned vertex. The algorithm for decomposing the graph
guarantees that $A\leq B$ means that $B$ occurs at a later stage
than $A$. This partial order can be represented in an {\em Assur
scheme} as in Figure~\ref{assurscheme}.  This partial order, with
the identifications needed for linkage composition,  can be used
to re-assemble the graph from its Assur components.
\begin{figure}[htb]
\centering
\includegraphics{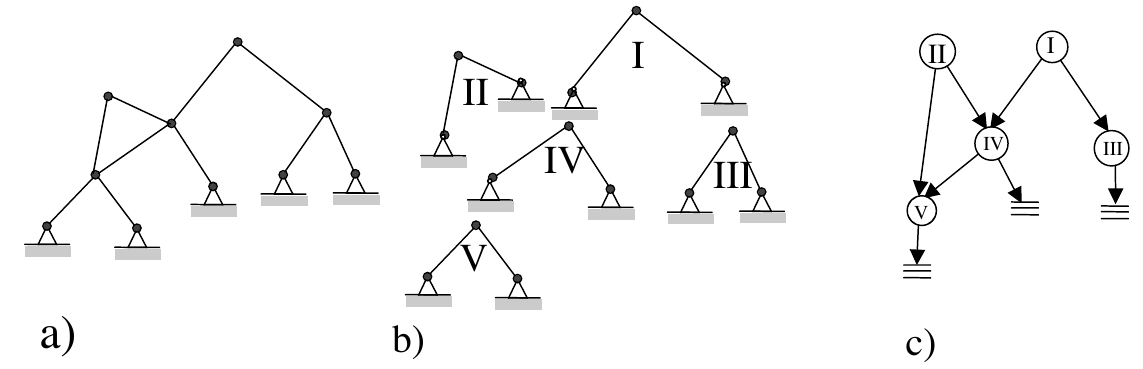}
\caption{An isostatics pinned framework a) has a unique decomposition into Assur graphs 
b) which is represented by a partial order or Assur scheme c).\
\label{assurscheme}}
\end{figure}

Replacing any edge in an isostatic framework produces a 1~DOF
linkage. The decomposition described in Theorem~\ref{POTheorem}
permits the analysis of this linkage in layers.  In fact, we can
place drivers in each Assur component to obtain linkages with
several degrees of freedom and their complex behavior can be
simply described by analyzing the individual Assur components.
This process of adding drivers is studied in more detail
in~\cite{SSW2}.


\subsection{ Generating Assur graphs}
We summarize inductive techniques to generate all Assur graphs.  
Engineers find such techniques of interest to
generate basic building blocks for synthesizing new linkages.
\begin{figure}[htb]
\centering
\includegraphics{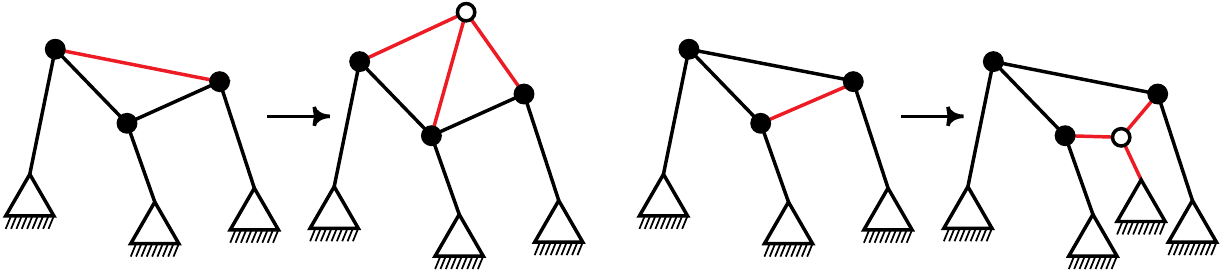}
\caption{An Edge-Split takes an Assur graph with at least four vertices to an Assur graph\
\label{assurwordfig06}}
\end{figure}
\begin{figure}[htb]
\centering
\includegraphics{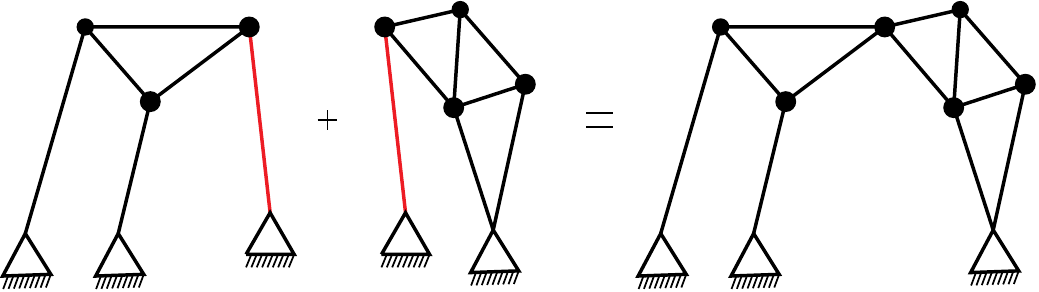}
\caption{2-sum of Assur graphs gives a new Assur graph with a removed pin.\label{palo01Fig}}
\end{figure}

The dyad is the only Assur graph on three vertices. There is no Assur graph on four
vertices. An Assur graph, whose corresponding rigidity circuit is $K_4$ is called a {\em basic} Assur graph.

To generate all Assur graphs (on five or more vertices) we use Theorem~\ref{CharacterThm}(ii) together
with Theorem~\ref{BJ} to generate all rigidity circuits. To get from a rigidity circuit $C$ to an Assur graph,
we choose a vertex $p^*$ of $C$ and split it into two or more pins (as in the proof of
Theorem~\ref{CharacterThm}).
The choice of $p^*$, the splitting of $p^*$ into a set $P$ of pins ($2\leq |P|\leq val(p^*)$), and choosing
for each edge incident to $p^*$ an endpoint from $P$ allows us to construct several Assur graphs from one
rigidity circuit, see Figure~\ref{PinRearrange}. We say that $G(I, P; E)$ and $G'(I, P', E)$ are
related by {\em pin rearrangement} if $G^* =G'^*$ (see Figure~\ref{PinRearrange}).

\begin{figure}[htb]
\centering
\includegraphics{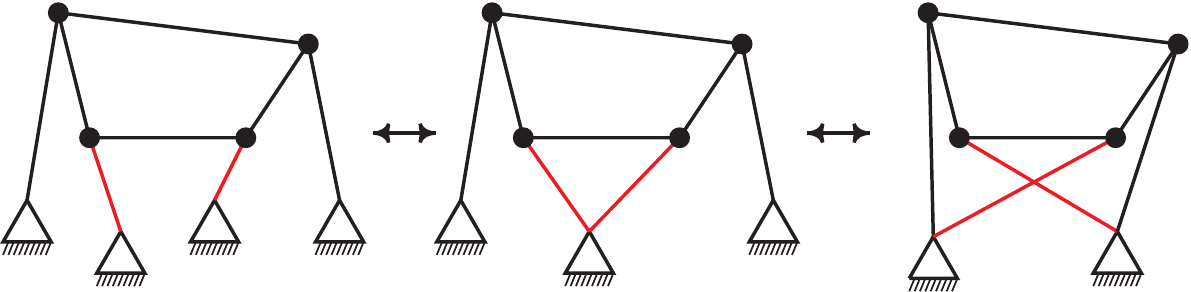}
\caption{Pin rearrangement (maintaining fact of at least two pins) \label{PinRearrange}}
\end{figure}

The operations of edge-split and 2-sum, which were used to generate rigidity circuits inductively,
can also be used directly on Assur graphs to generate new Assur graphs from old, see
Figures~\ref{assurwordfig06} and
~\ref{palo01Fig}.  In particular, the operation of 2-sum may be of practical value if, for space reasons for example,
a pinned vertex is to be eliminated, see Figure~\ref{palo01Fig}.

\begin{theorem}All Assur graphs on 5 or more vertices can be obtained from basic Assur graphs by
a sequence of edge-splits, pin-rearrangements and $2$-sums of smaller Assur graphs.
\end{theorem}

Since mechanical engineers might want to have additional tools readily
available for generating Assur graphs,
one can seek additional operations under which the class of Assur graphs is closed.
Vertex-split (creating two vertices of degree at least three) is another operation which takes a
rigidity circuit to a rigidity circuit,  and therefore takes
an Assur graph with at least three vertices to a larger Assur graph (Figure~\ref{assurwordfig10}).
This is, in a specific sense, the dual operation to edge-split~\cite{CMW}.
More generally, \cite{CMW} explores
a number of additional operations for generating larger circuits from smaller circuits.  Each of these
processes
will take an Assur graph to a larger Assur graph.
\begin{figure}[htb]
\centering
\includegraphics{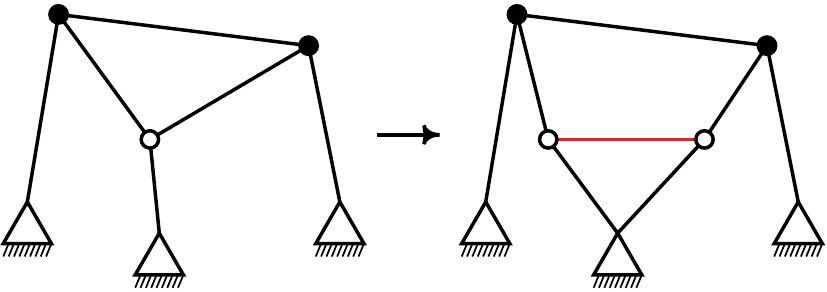}
\caption{Vertex split taking an Assur graph to an Assur graph. \label{assurwordfig10}}
\end{figure}

The inductive constructions for Assur graphs can be used to provide a  visual 
{\em certificate sequence} for an Assur graph.
If we are given a sequence of edge-splits and 2-sums starting from a dyad and ending with $G$,
see Figure~\ref{palo02Fig}, it is trivial to
verify that $G$ is an Assur graph. We constructed such a sequence in the proof of Theorem~\ref{PinThm}.
It is well known, see~\cite{TW}, that there are exponential algorithms to produce such a certificate.
 However, there are other
fast algorithms, for example the so called pebble games, 
see~\cite{JacobsHendrickson, pebblegame} to
detect all the rigidity properties
of graphs that can be adapted to verify the Assur property.

\begin{figure}[htb]
\centering
\includegraphics{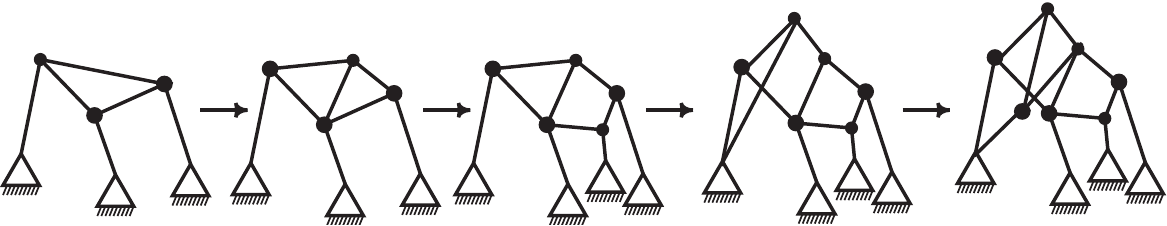}
\caption{Certificate sequence for the final Assur graph.\label{palo02Fig}}
\end{figure}

\section{Concluding comments}
The paper introduces, for the first time, the concept of Assur
graphs, in the rigorous mathematical terminology of rigidity 
theory. This work paves a new channel for cooperation between the
communities in kinematics and in rigidity theory. An example for such
channel is the material appearing in \S2.7 showing how to
transform determinate trusses used by the kinematicians  into
isostatic frameworks, widely employed by the rigidity theory
and structural engineering 
communities.

At this point, it is hard to predict all the practical
applications that are to benefit from this new relation between
the disciplines. Nevertheless, we anticipate practical results
from the use of rigidity theory in mechanisms as introduced in the
paper. Examples of such results, include  using rigid circuits
from rigidity theory to find the proper decomposition sequence of
pinned isostatic framework into Assur components (Section 3.2) and
generation of Assur graphs by applying two known operations to
their corresponding rigidity circuits (Section 3.3).

It is expected that new opportunities will be opened up, for
example, for mechanical engineers to comprehend topics in rigidity
theory that are used today in many disciplines, including biology,
communications and more. Mechanical engineers in the west may be
motivated to use the Assur graphs (Assur groups) concept as it is
widely applied in eastern Europe and Russia.

  Decomposing a larger linkage into Assur graphs and analyzing these pieces one at a time
is an effective way to analyze the overall motion, working in layers.  This paper has given a precise
mathematical foundation for the Assur method as well as a proof of its correctness and completeness.

We have followed standard engineering practice and developed the theory in the
language of bar-and-joint frameworks, but of course there is no need to replace a larger link (rigid body)
with an isostatic  bar-and-joint sub-framework to apply the counting techniques. It is simply convenient
to do so in order to streamline notation and graphics. All of our results may be reworded in
terms of body and bar structures or body and pin frameworks.
This presentation would be much closer to the original example in
Figures~1 and~2 and the counts of \S2.2.

To extend this type of analysis to  3D linkages, the lack of good characterization of isostatic bar-and-joint
frameworks in $3$D  is an initial obstacle.   However, if we focus on body-and-bar
or body-and-hinge structures (the analog of body-and-bar and body-and-pin frameworks in the plane)
then the generic DOF of these structures is computable by analogous counting 
techniques~\cite{WW2,Wh},
and all our combinatorial methods  will carry over sucessfully.

\bibliographystyle{plain}
\bibliography{assurcombArX}

\end{document}